\documentclass[reqno]{amsart}
\usepackage{amsmath, amssymb, amsthm, epsfig}
\usepackage{hyperref, latexsym}
\usepackage{url} 

\begin{document}
 \bibliographystyle{plain}

 \newtheorem{theorem}{Theorem}
 \newtheorem{lemma}[theorem]{Lemma}
 \newtheorem{proposition}[theorem]{Proposition}
 \newtheorem{corollary}[theorem]{Corollary}
 \theoremstyle{definition}
 \newtheorem{definition}[theorem]{Definition}
 \newtheorem{example}[theorem]{Example}
 \theoremstyle{remark}
 \newtheorem{remark}[theorem]{Remark}
 \newcommand{\mc}{\mathcal}
 \newcommand{\A}{\mc{A}}
 \newcommand{\B}{\mc{B}}
 \newcommand{\cc}{\mc{C}}
 \newcommand{\D}{\mc{D}}
 \newcommand{\E}{\mc{E}}
 \newcommand{\F}{\mc{F}}
 \newcommand{\G}{\mc{G}}
 \newcommand{\sH}{\mc{H}}
 \newcommand{\I}{\mc{I}}
 \newcommand{\J}{\mc{J}}
 \newcommand{\K}{\mc{K}}
 \newcommand{\lL}{\mc{L}}
 \newcommand{\M}{\mc{M}}
 \newcommand{\nn}{\mc{N}}
 \newcommand{\rr}{\mc{R}}
 \newcommand{\sS}{\mc{S}}
 \newcommand{\U}{\mc{U}}
 \newcommand{\X}{\mc{X}}
 \newcommand{\Y}{\mc{Y}}
 \newcommand{\C}{\mathbb{C}}
 \newcommand{\R}{\mathbb{R}}
 \newcommand{\N}{\mathbb{N}}
 \newcommand{\Q}{\mathbb{Q}}
 \newcommand{\Z}{\mathbb{Z}}
 \newcommand{\csch}{\mathrm{csch}}
 \newcommand{\tF}{\widehat{F}}
 \newcommand{\tG}{\widehat{G}}
 \newcommand{\tH}{\widehat{H}}
 \newcommand{\tf}{\widehat{f}}
 \newcommand{\ug}{\widehat{g}}
 \newcommand{\wg}{\widetilde{g}}
 \newcommand{\uh}{\widehat{h}}
 \newcommand{\wh}{\widetilde{h}}
 \newcommand{\wl}{\widetilde{l}}
 \newcommand{\tk}{\widehat{k}}
 \newcommand{\tK}{\widehat{K}}
 \newcommand{\tl}{\widehat{l}}
 \newcommand{\tL}{\widehat{L}}
 \newcommand{\tm}{\widehat{m}}
 \newcommand{\tM}{\widehat{M}}
 \newcommand{\tp}{\widehat{\varphi}}
 \newcommand{\tq}{\widehat{q}}
 \newcommand{\tT}{\widehat{T}}
 \newcommand{\tU}{\widehat{U}}
 \newcommand{\tu}{\widehat{u}}
 \newcommand{\tV}{\widehat{V}}
 \newcommand{\tv}{\widehat{v}}
 \newcommand{\tW}{\widehat{W}}
 \newcommand{\ba}{\boldsymbol{a}}
 \newcommand{\bal}{\boldsymbol{\alpha}}
 \newcommand{\bx}{\boldsymbol{x}}
 \newcommand{\p}{\varphi}
 \newcommand{\f}{\frac52}
 \newcommand{\g}{\frac32}
 \newcommand{\h}{\frac12}
 \newcommand{\hh}{\tfrac12}
 \newcommand{\ds}{\text{\rm d}s}
 \newcommand{\dt}{\text{\rm d}t}
 \newcommand{\du}{\text{\rm d}u}
 \newcommand{\dv}{\text{\rm d}v}
 \newcommand{\dw}{\text{\rm d}w}
 \newcommand{\dx}{\text{\rm d}x}
 \newcommand{\dy}{\text{\rm d}y}
 \newcommand{\dl}{\text{\rm d}\lambda}
 \newcommand{\dmu}{\text{\rm d}\mu(\lambda)}
 \newcommand{\dnu}{\text{\rm d}\nu(\lambda)}
\newcommand{\dnus}{\text{\rm d}\nu_{\sigma}(\lambda)}
 \newcommand{\dlnu}{\text{\rm d}\nu_l(\lambda)}
 \newcommand{\dnnu}{\text{\rm d}\nu_n(\lambda)}
\newcommand{\sech}{\text{\rm sech}}
 \def\today{\ifcase\month\or
  January\or February\or March\or April\or May\or June\or
  July\or August\or September\or October\or November\or December\fi
 \space\number\day, \number\year}

\title{Bounding $\zeta(s)$ in the critical strip}
\author[Emanuel Carneiro and Vorrapan Chandee]{Emanuel Carneiro and Vorrapan Chandee}

\date{\today}
\address{School of Mathematics, Institute for Advanced Study, Princeton, NJ 08540.}
\email{ecarneiro@math.ias.edu}
\address{Department of Mathematics, Stanford University, 450 Sierra Mall, Bldg. 380, Stanford, CA 94305}
\email{vchandee@math.stanford.edu}

\begin{abstract} Assuming the Riemann Hypothesis, we make use of the recently discovered \cite{CLV} extremal majorants and minorants of prescribed exponential type for the function $\log\left(\tfrac{4 + x^2}{(\alpha-1/2)^2 + x^2}\right)$ to find upper and lower bounds with explicit constants for $\log|\zeta(\alpha + it)|$ in the critical strip, extending the work of Chandee and Soundararajan \cite{CS}. 
\end{abstract}

\maketitle

\numberwithin{equation}{section}

\section{Introduction}
Littlewood showed in 1924 (see \cite{L}) that the Riemann Hypothesis (RH) implies a strong form of the Lindel\"{o}f Hypothesis, namely, on RH, for large real numbers $t$ there is a constant $C$ such that
\begin{equation}\label{Intro1}
 \left|\zeta\bigl(\tfrac{1}{2} + it\bigr)\right|\ll \exp\left( C \,\frac{\log t }{\log \log t}\right).
\end{equation}
Over the years no improvement has been made on the order of magnitude of the upper bound (\ref{Intro1}). The advances have rather focused on reducing the value of the admissible constant $C$ (see for instance the works by Ramachandra and Sankaranarayanan \cite{RS} and Soundararajan \cite{Sound}) and extending the results to general $L$-functions (see the work of Chandee \cite{Chandee}). A similar situation occurs when bounding the argument function $S(t) = \tfrac{1}{\pi} \arg \zeta\bigl(\tfrac{1}{2} + it\bigr)$, where the argument is defined by continuous variation along the line segments joining $2, 2 + it$ and $\tfrac{1}{2} + it$, taking the argument of $\zeta(s)$ at $2$ to be zero. Under RH, Littlewood showed that
\begin{equation}\label{Intro2}
 |S(t)| \ll \frac{\log t}{\log \log t}\,,
\end{equation}
and this bound has not been improved except for the size of the implied constant.

Recently, the idea of using the theory of extremal functions of exponential type was proved useful in both contexts, resulting in improved constants (and best up-to-date) for the upper bounds (\ref{Intro1}) and (\ref{Intro2}). The method of Goldston and Gonek \cite{GG} uses the explicit formula together with the classical Beurling-Selberg majorants and minorants of characteristic functions of intervals, and leads to the bound
\begin{equation*}\label{Intro3}
 |S(t)| \leq \bigl(\tfrac{1}{2} + o(1)\bigr) \frac{\log t}{\log \log t}\,.
\end{equation*}

In \cite{CS}, Chandee and Soundararajan recognized that the corresponding treatment for $\left|\zeta\bigl(\tfrac{1}{2} + it\bigr)\right|$, using the Hadamard's factorization and the explicit formula, would require the extremal minorant for the function $\log\left(\tfrac{4 + x^2}{x^2}\right)$, available in the framework of Carneiro and Vaaler \cite{CV2}. This combination was successful and led to the following bound \cite[Theorem 1]{CS}
\begin{equation}\label{Intro4}
 \left|\zeta\bigl(\tfrac{1}{2} + it\bigr)\right| \ll \exp \left( \frac{\log 2}{2} \frac{\log t}{\log \log t} + O\left( \frac{\log t \,\log \log \log t}{(\log \log t)^2}\right) \right)\,.
\end{equation}
It was mentioned in that paper that a similar approach to bounding $|\zeta(\alpha + it)|$, for $\alpha \neq 1/2$, would require the solution of the Beurling-Selberg extremal problem for the function
\begin{equation}\label{Intro5}
f_{\alpha}(x) = \log\left(\frac{4 + x^2}{(\alpha-1/2)^2 + x^2}\right)\,,
\end{equation}
which was not available at that particular time.

Very recently, Carneiro, Littmann and Vaaler in \cite{CLV} developed a new approach to the Beurling-Selberg extremal problem based on the solution for the Gaussian and tempered distribution arguments. With this method, they were able to extend the solution of this problem to a wide class of even functions, in particular, including the desired family (\ref{Intro5}). 

The purpose of this paper should be clear at this point. Here we make use of the recently discovered extremal majorants and minorants for $f_{\alpha}(x)$ to find upper and lower bounds with explicit constants for $|\zeta(\alpha + it)|$ on the critical strip. Observe that majorants for $f_{\alpha}(x)$ exist when $\alpha \neq 1/2$ and this is what makes the lower bounds possible. For simplicity, we will focus on the off-critical-line case (although the methods here plainly apply to the case $\alpha = 1/2$ with slightly different Fourier transform representations than those of \cite{CS}), assuming from now on that $\alpha = \alpha(t)$ is a real-valued function with $1/2 < \alpha \leq 1$. Since $|\zeta(\alpha + it)| = |\zeta(\alpha - it)|$ we might as well assume that $t \geq 0$. Our main results are the following.

\begin{theorem}[Upper Bound]\label{thm:upperbound} 
Assume RH. For large real numbers $t$, we have
\begin{equation*}
\log |\zeta(\alpha + it)| \leq  
\left\{ \begin{array}{l}
\log \left( 1 +  (\log t)^{1- 2\alpha}\right) \frac{\log t}{2\log \log t} + O\left( \frac{(\log t)^{2-2\alpha}}{(\log \log t)^2} \right), \\
\\
\ \ \ \ \ \ \ \ \ \ \ \ \ \ \ \ \ \ \ \ \ \ \ \ \ \ \ \ \ \ \ \ \ \ \ \ {\rm if}\,\,\, (\alpha - 1/2)\log \log t = O(1); \\
\\
\log(\log \log t) + O(1), \ \ \ \ \ \ \ \  {\rm if}\,\,\, (1- \alpha)\log \log t = O(1); \\
\\
\left( \frac{1}{2} + \frac{2\alpha - 1}{\alpha (1 - \alpha)} \right)\frac{(\log t)^{2 - 2\alpha}}{\log \log t}  + \log (2\log \log t) + O\left(\frac{(\log t)^{2 - 2\alpha}}{(1-\alpha)^2(\log \log t)^2}\right), \\
\ \ \ \ \ \ \ \ \ \ \ \ \ \ \ \ \ \ \ \ \ \ \ \ \ \ \ \ \ \ \ \ \ \ \ \ \ \ \ \ \ \ \ \ \ \ \ \ \ \ \ \ \ \ \ \ \ \ \ \ \ \ \ {\rm otherwise}.
\end{array}\right. 
\end{equation*}
\end{theorem}

\begin{theorem}[Lower Bound]\label{thm:lowerbound} 
Assume RH. For large real numbers $t$, we have
\begin{equation*}
\log |\zeta(\alpha + it)| \geq 
\left\{ \begin{array}{l}
 \log \left( 1 -  (\log t)^{1 - 2\alpha }\right) \frac{\log t}{2\log \log t} - O\left( \frac{(\log t)^{2-2\alpha}}{(\log \log t)^2(1 - (\log t)^{1 - 2\alpha})} \right),\\
\\
\ \ \ \ \ \ \ \ \ \ \ \ \ \ \ \ \ \ \ \ \ \ \ \ \ \ \ \ \ \ \ \ \ \ \ \ \ {\rm if}\,\,\, (\alpha - 1/2)\log \log t = O(1); \\
\\
  - \log (\log \log t) - O(1),   \ \ \ \ \ \ \ \ \ \, {\rm if}\,\,\, (1- \alpha)\log \log t = O(1); \\
\\
 - \left( \frac{1}{2} + \frac{2\alpha - 1}{\alpha (1 - \alpha)} \right)\frac{(\log t)^{2 - 2\alpha}}{\log \log t} - \log (2\log \log t) - O\left(\frac{(\log t)^{2 - 2\alpha}}{(1-\alpha)^2(\log \log t)^2}\right), \\
  \ \ \ \ \ \ \ \ \ \ \ \ \ \ \ \ \ \ \ \ \ \ \ \ \ \ \ \ \ \ \ \ \ \ \ \ \ \ \ \ \ \ \ \ \ \ \ \ \ \ \ \ \ \ \ \ \ \ \ \ \ \ \ \ \ {\rm otherwise}.
\end{array}\right. 
\end{equation*}
\end{theorem}

Observe that when $\alpha \to 1/2$ in Theorem \ref{thm:upperbound} we recover the main term of the bound (\ref{Intro4}). Also it is worth mentioning that the order of magnitude in the general upper bound in Theorem \ref{thm:upperbound} is a classical result in the theory of the Riemann zeta-function (see for instance \cite[Theorem 14.5]{T}), and the novelty here is in fact the method with which we arrive at this upper bound and the explicit computation of the implied constant.

With a refined calculation we can find the constant term when $\alpha = 1$ and obtain Littlewood's result \cite{L2} and \cite{L3} for bounds at $\text{Re}(s)= 1.$ 

\begin{corollary} \label{cor:boundAt1} 
Assume RH. For large real numbers $t$, we have
\begin{equation*}\label{Intro6}
 |\zeta(1 + it)| \leq (2e^{\gamma} + o(1)) \log \log t,
\end{equation*}
and
\begin{equation*}\label{Intro7}
 \frac{1}{|\zeta(1 + it)|} \leq \left(\frac{12e^{\gamma}}{\pi^2} + o(1)\right)\log \log t,
\end{equation*}
where $\gamma$ is the Euler constant.
\end{corollary}

The paper is divided in three sections plus an appendix. In Section 2 we prove the upper bounds for $\zeta(s)$ contained in Theorem \ref{thm:upperbound} and Corollary \ref{cor:boundAt1}. In Section 3 we prove the corresponding lower bounds for $\zeta(s)$ in Theorem \ref{thm:lowerbound} and Corollary \ref{cor:boundAt1}. In these two sections we will state the necessary facts concerning the extremal functions as supporting lemmas that will be ultimately proved in Section 4. The Appendix in the end details some of the asymptotic calculations carried along the proofs.

\section{Upper bound for $\zeta(s)$}
\subsection{Proof of Theorem \ref{thm:upperbound}} Let  
\begin{equation*}
 \xi(s) = s(1-s) \pi^{-s/2} \Gamma\left( \frac{s}{2}\right)\zeta(s)
\end{equation*}
be the Riemann's $\xi$-function. This function is an entire function of order 1 and satisfies the functional equation 
\begin{equation*}
 \xi(s) = \xi(1-s).
\end{equation*}
Hadamard's factorization formula gives us
\begin{equation*} \label{eqn:hadamard}
\xi(s) = e^{A+Bs} \prod_{\rho} \left(1-\frac{s}{\rho}\right)e^{s/\rho},
\end{equation*}
where $\rho = \tfrac 12 + i\gamma$ runs over the non-trivial zeros of $\zeta$, and from the functional equation we can show that $B= - \sum_{\rho} \text{Re} (1/\rho)$. On RH, $\gamma$ is real. 

By the functional equation and Hadamard's factorization formula, we obtain 
\begin{equation*}
\left|\frac{\xi(\alpha + it)}{\xi(5/2 - it)}\right| = \left|\frac{\xi(\alpha + it)}{\xi(-3/2 + it)}\right| = \prod_{\rho = 1/2 + i\gamma} \left( \frac{(\alpha - 1/2)^2 + (t-\gamma)^2}{4 + (t-\gamma)^2} \right)^{1/2}.
\end{equation*}
Recall Stirling's formula for the Gamma function \cite[Chapter 10]{Dav}
\begin{equation*}
\log \Gamma(z) = \frac{1}{2}\log 2\pi - z + \big(z - \tfrac 12\big)\log z + O\bigl(|z|^{-1}\bigr),
\end{equation*}
for large $|z|$. Using Stirling's formula and the fact that $|\zeta(5/2 - it)| \asymp 1$, we obtain 
\begin{equation} \label{eqn:logeqn}
\log |\zeta(\alpha + it)| = \left(\frac{5}{4} - \frac{\alpha}{2}\right) \log \frac t2 - \frac{1}{2}\sum_{\gamma} f_{\alpha}(t-\gamma) + O(1).
\end{equation}

The sum of $f_{\alpha}(t - \gamma)$ over the non-trivial zeros is hard to evaluate, so the key idea here is to replace $f_{\alpha}$ by its appropriate minorant (with a compactly supported Fourier transform) and then apply the following explicit formula which connects the zeros of the zeta-function and prime powers. The proof of the following lemma can be found in \cite[Theorem 5.12]{IK}.

\begin{lemma}[Explicit Formula] \label{lem:explicitformula}
Let $h(s)$ be analytic in the strip $|\text{\rm Im}(s)| \leq 1/2 + \epsilon$ for some $\epsilon > 0$, 
and such that $|h(s)| \ll (1+|s|)^{-(1+\delta)}$ for 
some $\delta >0$ when $|\text{\rm Re}(s)|\to \infty$.  Let $h(w)$ be real-valued for real $w$, and set $\widehat{h}(x) = \int_{-\infty}^{\infty} h(w) e^{-2\pi i x w} \> \dw$. 
Then 
\begin{align*}
\sum_{\rho} h(\gamma) = h\left(\frac{1}{2i} \right) &+h\left(-\frac{1}{2i}\right) 
- \frac{1}{2\pi} {\widehat h}(0) \log \pi  + \frac{1}{2 \pi}  \int_{-\infty}^{\infty} h(u) \,\text{\rm Re }\frac{\Gamma '}{\Gamma}\left(\frac{1}{4}   + \frac{iu}{2} \right) \> \du \\
& - \frac{1}{2\pi}\sum_{n=2}^{\infty} \frac{\Lambda(n)}{\sqrt{n}}\left(\widehat{h}\left( \frac{\log n}{2\pi }\right) + \widehat{h}\left( \frac{-\log n}{2\pi }\right)\right).
\end{align*}\end{lemma} 

The properties of the minorant function that we are interested in are described in the next lemma, that shall be proved in Section 4.

\begin{lemma}[Extremal Minorant] \label{prop:mainprop} 
Let $\Delta$ denote a positive real number. There is a unique entire function $g_{\Delta}$ which satisfies the following properties:
\begin{itemize}
 \item[(i)] For all real $x$ we have 
 \begin{equation}\label{minlemma21}
- \frac{C}{1+x^2} \le g_{\Delta}(x) \le f_{\alpha}(x), 
\end{equation}
for some positive constant $C$. For any complex number $x+iy$ we have 
\begin{equation}\label{minlemma2}
 |g_{\Delta}(x+iy)| \ll  \frac{\Delta^2  }{1+ \Delta |x+iy|}e^{2\pi \Delta|y|}.
\end{equation}
 
\item[(ii)] The Fourier transform of $g_{\Delta}$, namely 
 \begin{equation*}
 {\hat g}_{\Delta}(\xi) = \int_{-\infty}^{\infty} g_{\Delta}(x) e^{-2\pi i x \xi }\, \dx,  
 \end{equation*}
is a continuous real valued function supported on the interval $[-\Delta, \Delta]$. For $0 \leq |\xi|\leq \Delta$ it is given by:
\begin{align*}\label{eqn:GDelta}
\begin{split}
\hat{g}_{\Delta}(\xi) &= \sum_{k = 0}^{\infty} (-1)^k  \left( \frac{k + 1}{|\xi| + k \Delta} \left( e^{-2\pi (|\xi| + k\Delta)(\alpha - 1/2)} -  e^{-4\pi (|\xi| + k\Delta)}\right) \right. \\ 
&\ \  \ \ \ \left. - \frac{k + 1}{\Delta (k + 2) - |\xi| } \left( e^{2\pi(|\xi|-(k + 2)\Delta)(\alpha - 1/2)} - e^{4\pi (|\xi| - (k + 2)\Delta)} \right)\right).
\end{split}
\end{align*}
In particular, if $\xi = 0,$ we have
\begin{equation*}\label{FTatzero}
\hat{g}_{\Delta}(0) =  2\pi \left( \frac 52 - \alpha \right) - \frac{2}{\Delta}\log \left(\frac{1 + e^{-(2\alpha - 1)\pi\Delta}}{1 + e^{-4\pi\Delta}}\right).
\end{equation*}

\item[(iii)] The $L^1$-distance between $g_{\Delta}$ and $f_{\alpha}$ equals to
\begin{equation*}
 \int_{-\infty}^{\infty} \{f_{\alpha}(x)-g_{\Delta}(x)\}\,\dx = \frac{2}{\Delta} \Big(\log \bigl(1 + e^{-(2\alpha - 1)\pi \Delta} \bigr) - \log \bigl(1 + e^{-4\pi \Delta}\bigr)\Big) .
\end{equation*}
\end{itemize}
 \end{lemma}
 
From (\ref{eqn:logeqn}) and (i) of Lemma \ref{prop:mainprop} we obtain, for any $\Delta >0$,
\begin{equation}\label{Final1}
\log |\zeta(\alpha + it)| \leq  \left(\frac{5}{4} - \frac{\alpha}{2}\right) \log \frac t2 - \frac{1}{2}\sum_{\gamma} g_{\Delta}(t-\gamma) + O(1).
\end{equation}
To bound the sum of $g_{\Delta}(t - \gamma)$ we let $h(z) = g_{\Delta}(t - z)$ and apply Lemma \ref{lem:explicitformula} to get (observe that the growth condition  $|h(s)| \ll (1+|s|)^{-(1+\delta)}$ for some $\delta >0$ can be derived from (\ref{minlemma21}), or alternatively, directly from (\ref{eqn:sumofGDelta1}) and (\ref{relationforf}))

\begin{align}\label{expformulag}
\begin{split}
\sum_{\rho} g_{\Delta}(t-\gamma) =\Bigl\{ g_{\Delta}\Big(t&- \frac{1}{2i} \Big) +g_{\Delta}\Big(t+\frac{1}{2i}\Big) \Bigr\}
- \frac{1}{2\pi} {\widehat g}_{\Delta}(0) \log \pi  \\
&+ \frac{1}{2 \pi}  \int_{-\infty}^{\infty} g_{\Delta}(x) \,\text{\rm Re }\frac{\Gamma '}{\Gamma}\left(\frac{1}{4}   + \frac{i(t-x)}{2} \right)  \dx \\
& \ \ \ \ \ \ - \frac{1}{2\pi}\sum_{n=2}^{\infty} \frac{\Lambda(n)}{\sqrt{n}}\widehat{g}_{\Delta}\Big( \frac{\log n}{2\pi }\Big)\big( e^{-it \log n} + e^{it \log n}\big)
\end{split}
\end{align}
We now proceed to the asymptotic analysis of each of the elements on the right hand side of the expression (\ref{expformulag}).

\subsubsection{First term} From (i) of Lemma \ref{prop:mainprop} we get
\begin{equation}\label{firstterm}
\Bigl| g_{\Delta}\Big(t- \frac{1}{2i} \Big) +g_{\Delta}\Big(t+\frac{1}{2i}\Big) \Bigr| \ll \Delta^2 \frac{e^{\pi \Delta}}{1 + \Delta t}.
\end{equation}

\subsubsection{Second term} From (ii) of Lemma \ref{prop:mainprop} we get 
\begin{equation}\label{secondterm}
\frac{1}{2\pi} \widehat{g}_{\Delta}(0) \log \pi = \left( \frac 52 - \alpha \right)\log \pi - \frac{\log \pi}{\pi\Delta}\log \left(\frac{1 + e^{-(2\alpha - 1)\pi\Delta}}{1 + e^{-4\pi\Delta}}\right).
\end{equation} 

\subsubsection{Third term} We will now show that 
\begin{align}\label{tabound}
\begin{split}
\frac{1}{2 \pi}  \int_{-\infty}^{\infty} & g_{\Delta}(x) \,\text{\rm Re }\frac{\Gamma '}{\Gamma}\left(\frac{1}{4}   + \frac{i(t-x)}{2} \right) \dx\\
&= \left( \frac{5}{2} - \alpha \right) \log \frac{t}{2} - \frac{1}{\pi \Delta} \log \left(\frac{1 + e^{-(2\alpha - 1)\pi\Delta}}{1 + e^{-4\pi\Delta}}\right)\log\frac{t}{2} \\
& \ \ \ \ \ \ \ \ \ \ \ \ \ \ \ \ \ \ \ \ \ \ \ \ \ \ \ \ \ \ \ \ \ \ \ \ \ \ + O\left(\frac{\Delta \log(1 + \sqrt{t}\Delta)}{\sqrt{t}}\right).
\end{split}
\end{align}
From (i) of Lemma \ref{prop:mainprop}, for $x \neq 0$, we get
\begin{equation*}
- \frac{C}{1+x^2} \le g_{\Delta}(x) \le f_{\alpha}(x) \leq \frac{4}{x^2},
\end{equation*}
and hence 
\begin{equation}\label{gmin}
|g_{\Delta}(x)| \ll \min \left\{\frac{1}{x^2}, \frac{\Delta^2}{1+ \Delta |x|}\right\}.
\end{equation}
Since $\text{\rm Re }\tfrac{\Gamma '}{\Gamma}(\tfrac{1}{4}   + iu) \ll \log(|u| +2)$, we see that for sufficiently large $t$,
\begin{equation}\label{tabound1}
\int_{4\sqrt{t}}^{\infty} g_{\Delta}(x) \,\text{\rm Re }\frac{\Gamma '}{\Gamma}\left(\frac{1}{4}   + \frac{i(t-x)}{2} \right) \dx \ll \int_{4\sqrt{t}}^{\infty}\frac{\log (x+2)}{x^2}\, \dx \ll \frac{\log t}{\sqrt{t}}.
\end{equation}
By similar arguments, 
\begin{equation}\label{tabound2}
\int^{-4\sqrt{t}}_{-\infty} g_{\Delta}(x) \, \text{\rm Re }\frac{\Gamma '}{\Gamma}\left( \frac{1}{4} + \frac{i(t-x)}{2} \right) \dx\ll \frac{\log t}{\sqrt{t}}.
\end{equation}
Finally, we use that
\begin{equation*}
 \frac{\Gamma'(s)}{\Gamma(s)}=\log s + O\bigl(|s|^{-1}\bigr)
\end{equation*}
for large $s$, together with (\ref{gmin}), part (iii) of Lemma \ref{prop:mainprop}, and the fact that 
$\int_{-\infty}^{\infty} f_\alpha(x)\,\dx = 2\pi \left( \frac{5}{2} - \alpha \right) $, to get

\begin{align}\label{tabound3}
\begin{split}
\int_{-4\sqrt{t}}^{4\sqrt{t}} g_{\Delta}(x) &\,\text{\rm Re }\frac{\Gamma '}{\Gamma}\left( \frac{1}{4} + \frac{i(t-x)}{2} \right)\dx \\
& = \int_{-4\sqrt{t}}^{4\sqrt{t}} g_{\Delta}(x) \left( \log \frac t2 + O\left(\frac{1}{\sqrt t}\right)\right) \dx\\
&= \log \frac t2 \int_{-\infty}^{\infty} g_{\Delta}(x)\, \dx  + O\left( \frac{\Delta \log (1 + \sqrt{t}\Delta)}{\sqrt{t}} \right)\\
&= 2\pi \left( \frac{5}{2} - \alpha \right) \log \frac t2 - \frac{2}{\Delta}\log \left( \frac{1 + e^{-(2\alpha - 1)\pi \Delta}}{1 + e^{-4\pi\Delta}}\right) \log \frac t2 \\
& \ \ \ \ \ \ \ \ \ \ \ \ \ \ \ \ \ \ \ \ \ \ \ \ \ \ \ \ \ \ \ \ \ \ \ \ \ \ \ \ \ \ \ \ + O\left( \frac{\Delta \log (1 + \sqrt{t}\Delta)}{\sqrt{t}} \right).
\end{split}
\end{align}
Combining (\ref{tabound1}), (\ref{tabound2}) and (\ref{tabound3}) we arrive at (\ref{tabound}).

\subsubsection{Fourth term (sum over prime powers)} This is the hardest term to analyze. We will have to make use of the explicit expression for Fourier transform of $g_{\Delta}$ described in (ii) of Lemma \ref{prop:mainprop} to get
\begin{align} \label{fabound}
\begin{split}
& \frac{1}{2\pi} \sum_{n = 2}^{\infty} \frac{\Lambda(n)}{\sqrt{n}} \hat{g}_{\Delta} \left( \frac{\log n}{2\pi} \right)\left( e^{-i t\log n} + e^{it \log n }\right) \\
&\\
&= \sum_{n \leq e^{2\pi \Delta}} \frac{\Lambda(n)}{\sqrt{n}} \sum_{k = 0}^{\infty} \left(\frac{k + 1}{\log n + 2\pi k \Delta}  \frac{e^{-(2\alpha - 1)\pi k\Delta}}{n^{\alpha - 1/2}} \right.\\
&\ \ \ \ \ \ \ \ \   \left.  - \frac{k + 1}{(2\pi\Delta (k + 2) - \log n)}  \frac{n^{\alpha - 1/2}}{e^{(2\alpha - 1)\pi (k + 2) \Delta}} \right) (-1)^k\left( e^{-i t\log n} + e^{it \log n }\right) \\
&\ \ \ \ \ \ \ \ \ \ \ \ \ \ \ \ \ \ \ \ \ \ \ \ \ \ \ \ \ \ \ \ \ \ \ \ \ \ \ \ \ \ \ \ \ \ - 2 {\rm Re} \sum_{n \leq e^{2\pi \Delta}} \frac{\Lambda(n)}{n^{5/2 + it}\log n} + O\bigl(e^{- 3\pi \Delta}\bigr) \\
&\\
&= \sum_{n \leq e^{2\pi \Delta}} \frac{\Lambda(n)}{\sqrt{n}} \sum_{k = 0}^{\infty} \left(\frac{k + 1}{\log n + 2\pi k \Delta}  \frac{e^{-(2\alpha - 1)\pi k\Delta}}{n^{\alpha - 1/2}} \right.\\
&\ \ \ \ \ \ \ \ \  \left.  - \frac{k + 1}{(2\pi\Delta (k + 2) - \log n)}  \frac{n^{\alpha - 1/2}}{e^{(2\alpha - 1)\pi (k + 2) \Delta}} \right) (-1)^k\left( e^{-i t\log n} + e^{it \log n }\right) \\
&\ \ \ \ \ \ \ \ \ \ \ \ \ \ \ \ \ \ \ \ \ \ \ \ \ \ \ \ \ \ \ \ \ \ \ \ \ \ \ \ \ \ \ \ \ \ - 2 \log \left|\zeta \left(\frac 52 + it\right)\right| + O\bigl(e^{- 3\pi \Delta}\bigr). 
\end{split}
\end{align}
From now on we let $x = e^{2\pi \Delta}.$ 
Since $\frac{1}{\log y} \frac{1}{y^{\alpha - 1/2}}$ is a non-increasing function for $y >1$, we deduce that for all integers $k \geq 0$ (recall that $n \leq x$), 
\begin{equation} \label{ineq:twotermdifferent}
\frac{1}{\log nx^k}  \frac{1}{(nx^k)^{\alpha - 1/2}}  - \frac{1}{(\log \tfrac{x^{k + 2}}{n})}  \frac{1}{\big(\tfrac{x^{k + 2}}{n}\big)^{\alpha - 1/2}} \geq 0.
\end{equation}

The following two lemmas will be used to bound the sum over prime powers above.
\begin{lemma} \label{lem:initialIneq} For all $k \geq 0$ and $n \leq x,$ 
\begin{align*} 
 (k+ 1) &\left( \frac{1}{\log n x^k} \frac{1}{(nx^k)^{\alpha - 1/2}} - \frac{1}{\log \tfrac{x^{k + 2}}{n}}  \frac{1}{\big( \tfrac{x^{k + 2}}{n}\big)^{\alpha - 1/2}} \right) \\
&\geq (k+ 2) \left( \frac{1}{\log n x^{k +1}} \frac{1}{(nx^{k + 1})^{\alpha - 1/2}} - \frac{1}{\log \tfrac{x^{k + 3}}{n}}  \frac{1}{\big( \tfrac{x^{k + 3}}{n}\big)^{\alpha - 1/2}} \right).
\end{align*}
\end{lemma}
\begin{proof}
The above inequality is equivalent to
\begin{align*} 
 \frac{k + 1}{x^{(\alpha - 1/2)k}} &\left( \frac{1}{n^{\alpha - 1/2}{\log n x^k}} - \frac{n^{\alpha - 1/2}}{x^{2\alpha - 1} \log \tfrac{x^{k + 2}}{n}} \right) \\
& \geq \frac{k + 2}{x^{(\alpha - 1/2)(k + 1)}} \left( \frac{1}{n^{\alpha - 1/2}{\log n x^{k + 1}}} - \frac{n^{\alpha - 1/2}}{x^{2\alpha - 1} \log \tfrac{x^{k + 3}}{n}} \right).
\end{align*}
Since $\frac{1}{x^{\alpha - 1/2}} \leq 1,$ it suffices to show that
\begin{align*}
\frac{1}{n^{\alpha - 1/2}}& \left( \frac{k + 1}{k\log x + \log n } - \frac{k + 2}{(k + 1)\log x + \log n } \right)\\ &\geq  \frac{n^{\alpha - 1/2}}{x^{2\alpha - 1}}\left( \frac{k + 1}{(k + 2)\log x - \log n } - \frac{k + 2}{(k + 3)\log x - \log n } \right).
\end{align*}
The above is true since 
\begin{align*}
 \frac{k + 1}{k\log x + \log n }& - \frac{k + 2}{(k + 1)\log x + \log n } \\
& = \frac{\log x - \log n}{(k\log x + \log n)((k + 1)\log x + \log n)} \geq 0,
\end{align*}
while 
\begin{align*}
\frac{k + 1}{(k + 2)\log x - \log n }& - \frac{k + 2}{(k + 3)\log x - \log n} \\
& = \frac{\log n - \log x}{((k + 2)\log x - \log n)((k + 3)\log x - \log n)} \leq 0. 
\end{align*}
\end{proof}

\begin{lemma} \label{lem:boundSumOverk}
For all $k \geq 1$ and positive real numbers $2 \leq n \leq x,$ 
\begin{equation*} \label{ineq:boundFirstSum}
\frac{1}{\log x} - \frac{1}{x^{\alpha - 1/2}\log x} \leq \frac{k + 1}{k \log x + \log n} - \frac{k + 2}{x^{\alpha - 1/2}((k + 1) \log x + \log n)},  
\end{equation*}
and
\begin{equation*} \label{ineq:boundSecondSum}
\frac{1}{\log x} - \frac{1}{x^{\alpha - 1/2}\log x} \geq 
\frac{k + 1}{(k + 2)\log x - \log n} - \frac{k + 2}{x^{\alpha - 1/2}((k + 3) \log x - \log n)}.
\end{equation*}
\end{lemma}

\begin{proof} We will only show the proof for one inequality. The proof of the other is quite similar. Let us show that
\begin{equation*}
\frac{1}{\log x} - \frac{1}{x^{\alpha - 1/2}\log x} \leq \frac{k + 1}{k \log x + \log n} - \frac{k + 2}{x^{\alpha - 1/2}((k + 1) \log x + \log n)}
\end{equation*}
This is equivalent to
\begin{eqnarray*}
\frac{1}{x^{\alpha - 1/2}}\left(  \frac{k + 2}{(k + 1) \log x + \log n} - \frac{1}{\log x}\right) &\leq& \frac{k + 1}{k \log x + \log n} - \frac{1}{\log x} \\
\iff \frac{1}{x^{\alpha - 1/2}} \left( \frac{\log x - \log n}{((k + 1) \log x + \log n)\log x}\right) &\leq& \left( \frac{\log x - \log n}{(k\log x + \log n)\log x}\right).
\end{eqnarray*}
The above inequality follows from the fact that $\frac{1}{x^{\alpha - 1/2}} \leq 1$. This proves the lemma. 
\end{proof}

From (\ref{fabound}), (\ref{ineq:twotermdifferent}) and Lemma \ref{lem:initialIneq}, we have
\begin{align} \label{eqn:mainsumoverprimes}
\begin{split}
\frac{1}{2\pi} & \sum_{n = 2}^{\infty} \frac{\Lambda(n)}{\sqrt{n}} \hat{g}_{\Delta} \left( \frac{\log n}{2\pi} \right)\left( e^{-i t\log n} + e^{it \log n }\right)  \\
&\leq  2 \sum_{n \leq x} \frac{\Lambda(n)}{\sqrt{n}} \sum_{k = 0}^{\infty} (-1)^k \left(\frac{k + 1}{k \log x + \log n }  \frac{1}{(nx^k)^{\alpha - 1/2}}\right.  \\
&   \ \ \ \ \ \ \ \ \ \ \ \ \ \ \ \ \ \ \ \ \ \ \ \ \ \ \left. - \frac{k + 1}{((k + 2)\log x - \log n)}  \frac{n^{\alpha - 1/2}}{(x^{k + 2})^{\alpha - 1/2}} \right)  \\
& \ \ \ \ \ \ \ \ \ \ \ \ \ \ \ \ \ \ \ \ \ \ \ \ \ \ \ \ \ \ \ \ \ \ \ \ \ \ - 2 \log \left|\zeta \left(\frac 52 + it\right)\right| + O(e^{- 3\pi \Delta}). 
\end{split}
\end{align} 
Rearranging the terms and using Lemma \ref{lem:boundSumOverk}, we obtain that the sum over $k$ is bounded above by
\begin{align} \label{ineq:BoundOverk}
\begin{split}
\sum_{k = 0}^{\infty}& (-1)^k \left(\frac{k + 1}{k \log x + \log n }  \frac{1}{(nx^k)^{\alpha - 1/2}}  - \frac{k + 1}{((k + 2)\log x - \log n)}  \frac{n^{\alpha - 1/2}}{(x^{k + 2})^{\alpha - 1/2}} \right) \\
&\leq \frac{1}{n^{\alpha - 1/2}\log n} - \frac{n^{\alpha - 1/2}}{(2\log x - \log n)\,x^{2\alpha - 1}} \\
&  \ \ \ \  \ \ \ \ \ \ \ \ \ \ \ \ \ \ \ \ \ \ \ \ \ \ + \frac{1}{\log x} \sum_{k = 1}^{\infty} (-1)^k \left(  \frac{1}{(nx^k)^{\alpha - 1/2}}  - \frac{n^{\alpha - 1/2}}{(x^{k + 2})^{\alpha - 1/2}} \right)  \\
&=\frac{1}{n^{\alpha - 1/2}\log n} - \frac{n^{\alpha - 1/2}}{(2\log x - \log n)\,x^{2\alpha - 1}} \\
&  \ \ \ \ \ \ \ \ \ \ \ \ \ \ \ \ \ \ \ \ \ \ \ \ \ \ \ - \frac{1}{\log x(x^{\alpha - 1/2} + 1)} \left( \frac{1}{n^{\alpha - 1/2}} - \frac{n^{\alpha - 1/2}}{x^{2\alpha - 1}} \right). 
\end{split}
\end{align}

Recall that the prime number theorem on the Riemann hypothesis is
\begin{equation} \label{eqn:PNT}
\sum_{n \leq x} \Lambda (n) = x + O(x^{1/2}\log x).
\end{equation}
Therefore using partial summation, (\ref{ineq:BoundOverk}), and ({\ref{eqn:PNT}}), we obtain that
\begin{align*}
 &\sum_{n \leq x} \frac{\Lambda(n)}{\sqrt{n}}  \sum_{k = 0}^{\infty} (-1)^k \left(\frac{k + 1}{k \log x + \log n }  \frac{1}{(nx^k)^{\alpha - 1/2}}  - \frac{k + 1}{((k + 2)\log x - \log n)}  \frac{n^{\alpha - 1/2}}{(x^{k + 2})^{\alpha - 1/2}} \right) \\
&\leq  \int_{2}^{x} \left( \frac{1}{t^{\alpha} \log t} - \frac{1}{\log x(x^{\alpha - 1/2} + 1) t^{\alpha}} \right) \dt\\
& \ \ \ \ \ \ \  \ \ \ \  \ \ \ \ \ \ \  - \frac{1}{x^{2\alpha - 1}} \int_{2}^{x} \left( \frac{1}{t^{1- \alpha}(2\log x - \log t)}  - \frac{1}{t^{1 - \alpha}\log x(x^{\alpha - 1/2} + 1)} \right)  \, \dt \\
&  \ \ \ \ \ \ \ \ \ \ \ \ \ \ \ \ \ \ \ \ \ \ \ \ \ \ \ \ \  \ \ \ \ \ \ \ \ \ \ \ \ \ \ \ \ \ \ \ \ \ \ \ \ + O\left(\min\left\{\log x, \frac{1}{\alpha - 1/2} \right\}\right) \\
& = A(x) - B(x) + O\left(\min\left\{\log x, \frac{1}{\alpha - 1/2} \right\}\right),
\end{align*}
where the asymptotics af $A(x)$ and $B(x)$ (calculated in the Appendix) are given by
\begin{equation*}
A(x) = \left\{ \begin{array}{l} \log \log x  + O\left(1\right), \ \ \ \   {\rm if }\,\,\, (1 - \alpha)\log x = O(1); \\
\frac{x^{1-\alpha}}{(1-\alpha)\log x} \left(\frac{x^{\alpha - 1/2}}{x^{\alpha - 1/2} + 1}\right) + \log \log x + O\left(\frac{x^{1 - \alpha}}{(1- \alpha)^2\log^2 x}\right), \ \    {\rm otherwise}, \end{array}\right.
\end{equation*}
and
\begin{equation*}
B(x) = \frac{1}{\alpha} \frac{x^{1-\alpha}}{\log x} \left( \frac{x^{\alpha - 1/2}}{x^{\alpha-1/2} +1}\right) + O\left( \frac{x^{1-\alpha}}{\log^2 x}\right).
\end{equation*}
Therefore the sum over prime powers is 
\begin{align}
\begin{split}
\label{eqn:summarysumoverprime}
& \sum_{n \leq x} \frac{\Lambda(n)}{\sqrt{n}} \sum_{k = 0}^{\infty} (-1)^k \left(\frac{k + 1}{k \log x + \log n }  \frac{1}{(nx^k)^{\alpha - 1/2}}  - \frac{k + 1}{((k + 2)\log x - \log n)}  \frac{n^{\alpha - 1/2}}{(x^{k + 2})^{\alpha - 1/2}} \right) \\
&\leq \left\{ \begin{array}{l} \log \log x  + O\left(1\right), \ \ \  {\rm if }\,\,\, (1 - \alpha)\log x = O(1);  \\
\frac{2\alpha - 1}{\alpha (1 - \alpha)}\frac{x^{1-\alpha}}{\log x} \left(\frac{x^{\alpha - 1/2}}{x^{\alpha - 1/2} + 1}\right) + \log \log x + O\left(\frac{x^{1 - \alpha}}{(1- \alpha)^2\log^2 x}\right), \ \  {\rm otherwise.} \end{array}\right. 
\end{split}
\end{align}

\subsubsection{Final Analysis} Combining all the results above (equations (\ref{Final1})-(\ref{tabound}) and (\ref{eqn:summarysumoverprime})), and recalling that $x = e^{2\pi \Delta}$, we obtain 
\begin{align}\label{Final2}
\begin{split}
\log &|\zeta(\alpha + it)| \leq \frac{1}{2\pi\Delta} \log \left( \frac{1 + e^{-(2\alpha - 1)\pi \Delta}}{1 + e^{-4\pi\Delta}}\right) \log \frac t2 \\
&\\
& \ \ \ \ \ \ \ \ \ \ \ \ \ \ \ \ \ \ \ \ \ \  + O\left( \Delta^2 \frac{e^{\pi \Delta}}{1 + \Delta t} \right) + O\left(\frac{\Delta \log(1 + \sqrt{t}\Delta)}{\sqrt{t}}\right)\\
&\\
& \ \ \ \ + \left\{ \begin{array}{l} \log 2\pi \Delta  + O\left(1\right), \ \  {\rm if }\,\,\, (1 - \alpha)\Delta = O(1);  \\
\\
\frac{2\alpha - 1}{\alpha (1 - \alpha)}\frac{e^{(2-2\alpha)\pi\Delta}}{2\pi\Delta} \left(\frac{e^{(2\alpha - 1)\pi\Delta}}{e^{(2\alpha - 1)\pi\Delta} + 1}\right) + \log 2\pi\Delta + O\left(\frac{e^{(2 - 2\alpha)\pi\Delta}}{(1- \alpha)^2 \pi^2\Delta^2}\right), \\
\ \ \ \ \ \ \ \ \ \ \ \ \ \ \ \ \ \ \ \ \ \ \ \ \ \ \ \ \ \ \ \ \ \ \ \ \ \ \ \ \ \ \ \ \ \ \ \ \ \ \ \ \ \ \ \ \ \ \ \ \ \ \ \ \ \ {\rm otherwise.} \end{array}\right.
\end{split}
\end{align}
An optimal bound in (\ref{Final2}) occurs when $\pi \Delta = \log \log t.$ This upper bound depends on how far $\alpha$ is from 1/2 and 1, and we examine three cases:

\begin{enumerate}
\item[{\it Case 1}.]$\alpha - 1/2 = O\left(\frac{1}{\log \log t}\right)$.\\
For this case, $\frac{2\alpha - 1}{\alpha (1-\alpha)} = O\left(\frac{1}{\log \log t}\right)$, and the upper bound (\ref{Final2}) becomes
\begin{equation*}
\log |\zeta(\alpha + it)|  \leq  \log \left( 1 +  (\log t)^{-(2\alpha - 1)}\right) \frac{\log t}{2\log \log t} + O\left( \frac{(\log t)^{2-2\alpha}}{(\log \log t)^2} \right),
\end{equation*}
which, as mentioned in the Introduction, recovers the main term in (\ref{Intro4}) when $\alpha \rightarrow 1/2$.\\

\item[\it Case 2.] $1 - \alpha = O\left(\frac{1}{\log \log t}\right).$ \\
The upper bound is
$$ \log |\zeta(\alpha + it)|  \leq \log (2 \log \log t) + O(1).$$
In \S 2.2, we will bound explicitly what the constant term is for $|\zeta(1 + it)|.$\\

\item[\it Case 3.] Otherwise, we have 
\begin{equation*}
\log \left( 1 +  e^{-(2\alpha - 1)\pi \Delta}\right) \asymp \frac{1}{(\log t)^{2\alpha - 1} },
\end{equation*}
and the upper bound (\ref{Final2}) becomes
\begin{align*}
\log |\zeta(\alpha + it)|  \leq &\left( \frac{1}{2} + \frac{2\alpha - 1}{\alpha (1 - \alpha)} \right)\frac{(\log t)^{2 - 2\alpha}}{\log \log t} \\
& \ \ \ \ \ \ \ \ \ \ + \log (2\log \log t) + O\left(\frac{(\log t)^{2 - 2\alpha}}{(1 - \alpha)^2(\log \log t)^2}\right).
\end{align*}
\end{enumerate}
This completes the proof of Theorem \ref{thm:upperbound}.

\subsection{An upper bound for $|\zeta (1 + it)|$}  In this section we will bound $\zeta(1 + it)$ and rederive Littlewood's result \cite{L2}, which is
$$ |\zeta(1 + it)| \leq (2e^{\gamma} + o(1))\log \log t ,$$
where $\gamma$ is the Euler constant.

The method used to bound $|\zeta(1 + it)|$ is the same as the above except that we will bound $\sum_{n \leq x} \frac{\Lambda(n)}{n\log n}$ with an error term $o(1).$ From \cite{L2} and Merten's formula \cite{MV}, we have
$$  \sum_{n \leq x} \frac{\Lambda(n)}{n\log n} = \log \log x + \gamma + O\left(\frac{1}{\log x}\right).$$
Moreover by the prime number theorem (\ref{eqn:PNT}),
$$ \sum_{n \leq x} \frac{\Lambda(n)}{x(2\log x - \log n)} \leq \frac{1}{x \log x} \sum_{n \leq x} \Lambda(n) = O\left(\frac{1}{\log x}\right).$$
To obtain a constant term of the upper bound for $|\zeta(1 + it)|,$ we will exploit a refined upper bound for (\ref{eqn:logeqn}). To be precise, we can show that
\begin{align} \label{eqn:upplogeqn}
\begin{split}
\log \left|\zeta\big(\alpha + it\big)\right| &= \left(\frac{5}{4} - \frac{\alpha}{2}\right) \log \frac{t}{2} - \frac{1}{2}\sum_{\gamma} f_{\alpha}(t-\gamma) \\
& \ \ \ \ \  \ \ \ \ \ + \log|\zeta \big( \tfrac 52 + it \big)| - \left(\frac{5}{4} - \frac{\alpha}{2}\right)\log \pi + O\left(\frac{1}{t}\right)
\end{split}
\end{align}
using Stirling's formula.

Therefore by (\ref{eqn:upplogeqn}) and the bounds for each terms in the explicit formula, we obtain that
\begin{align*}
\log |\zeta(1 + it)| \,\, &\leq \,\, \log 2e^{\gamma}\pi \Delta + \frac{1}{2\pi\Delta}\log \left( \frac{1 + e^{-\pi \Delta}}{1 + e^{-4\pi\Delta}}\right) \log \frac t2 \\
& \,\,\,\,\,\,\,\,\,\,\,\,\,\,\,\, + O\left(\frac{\Delta \log (1 + \sqrt{t}\Delta)}{\sqrt{t}} + \frac{\Delta e^{\pi\Delta}}{t} + \frac{1}{\pi \Delta} + \frac{1}{t}\right).
\end{align*}
The upper bound of $|\zeta(1 + it)|$ in Corollary \ref{cor:boundAt1} follows from choosing  $\pi\Delta = \log \log t.$

\section{Lower bound for $\zeta(s)$}
\subsection{Proof of Theorem \ref{thm:lowerbound}}The method of computing a lower bound for $\zeta(\alpha + it)$ is similar to the one for the upper bound in Section 2, with the only difference being the use of a majorant function instead. The majorant function that we are interested in satisfies the following properties (that shall be proved in the next section).

\begin{lemma}[Extremal Majorant] \label{prop:mainproplowerbound} 
Let $\Delta$ denote a positive real number. There is a unique entire function $m_{\Delta}$ which satisfies the following properties: 
\begin{itemize}
 \item[(i)] For all real $x$ we have 
\begin{equation*}
f_{\alpha}(x) \leq  m_{\Delta}(x) \leq C \frac{1}{1+x^2} \,,
\end{equation*}
for some positive constant $C$.  For any complex number $x+iy$ we have 
\begin{equation*}
 |m_{\Delta}(x+iy)| \ll  \frac{\Delta^2  }{1+ \Delta |x+iy|}e^{2\pi \Delta|y|}.
\end{equation*}
 
 \item[(ii)]  The Fourier transform of $m_{\Delta}$, namely 
 $$
 {\hat m}_{\Delta}(\xi) = \int_{-\infty}^{\infty} m_{\Delta}(x) e^{-2\pi i x \xi } \dx,  
 $$ 
 is a continuous real valued function supported on the interval $[-\Delta, \Delta].$ For $0 \leq |\xi| \leq \Delta,$ it is given by
\begin{align*}\label{eqn:mDelta}
\hat{m}_{\Delta}(\xi) &= \sum_{k = 0}^{\infty}  \left( \frac{k + 1}{|\xi| + k \Delta} \left( e^{-2\pi (|\xi| + k\Delta)(\alpha - 1/2)} -  e^{-4\pi (|\xi| + k\Delta)}\right) \right. \\ 
&\ \ \ \ \ \ \left. - \frac{k + 1}{\Delta (k + 2) - |\xi| } \left( e^{2\pi(|\xi|-(k + 2)\Delta)(\alpha - 1/2)} - e^{4\pi (|\xi| - (k + 2)\Delta)} \right)\right).
\end{align*}
In particular,  if $\xi = 0,$ we have
\begin{equation*}
\hat{m}_{\Delta}(0) =  2\pi \left( \frac 52 - \alpha \right) - \frac{2}{\Delta}\log \left(\frac{1 - e^{-(2\alpha - 1)\pi\Delta}}{1 - e^{-4\pi\Delta}}\right).
\end{equation*}
 
\item[(iii)] The $L^1$-distance between $m_{\Delta}$ and $f_{\alpha}$ equals to
 $$
\int_{-\infty}^{\infty} \{m_{\Delta} (x) - f_{\alpha}(x)\}\,  \dx = \frac{2}{\Delta} \left(\log \bigl(1 - e^{-4\pi \Delta}\bigr) - \log \bigl(1 - e^{-(2\alpha - 1)\pi \Delta} \bigr)\right) .
 $$
\end{itemize}
 \end{lemma}

From (\ref{eqn:logeqn}) and (i) of Lemma \ref{prop:mainproplowerbound} we obtain, for any $\Delta >0$,
\begin{equation}\label{Final12}
\log |\zeta(\alpha + it)| \geq  \left(\frac{5}{4} - \frac{\alpha}{2}\right) \log \frac t2 - \frac{1}{2}\sum_{\gamma} m_{\Delta}(t-\gamma) + O(1).
\end{equation}
We then apply the explicit formula (Lemma \ref{lem:explicitformula}) to the majorant function $m_{\Delta}(z)$ to get
\begin{align}\label{Fai_1}
\begin{split}
\sum_{\rho} m_{\Delta}(t - \gamma) &= \left\{m_{\Delta}\Bigl( t - \frac{1}{2i}\Bigr) + m_{\Delta}\Bigl( t + \frac{1}{2i}\Bigr)\right\} - \frac{1}{2\pi} \hat{m}_{\Delta}(0) \log \pi \\
& \ \ \ \ + \frac{1}{2\pi}\int_{-\infty}^{\infty} m_{\Delta}(x)\, {\rm Re} \frac{\Gamma '}{\Gamma}\left( \frac{1}{4} + \frac{i(t-x)}{2} \right)\, \dx \\
& \ \ \ \ \ \ \ \ \ \ - \frac{1}{2\pi} \sum_{n = 2}^{\infty} \frac{\Lambda(n)}{\sqrt{n}}\, \hat{m}_{\Delta} \left( \frac{\log n}{2\pi} \right)\left( e^{-i t\log n} + e^{it \log n }\right),
\end{split}
\end{align}
and the asymptotic analysis follows just as before:
\subsubsection{First term} From (i) of Lemma \ref{prop:mainproplowerbound} we have
\begin{equation} \label{eqn:constantTtermLowerbound}
\left|m_{\Delta}\left( t - \frac{1}{2i}\right) + m_{\Delta}\left( t + \frac{1}{2i}\right)\right| \ll \Delta^2 \frac{e^{\pi \Delta}}{1+\Delta t}.
\end{equation}

\subsubsection{Second term} From (ii) of Lemma \ref{prop:mainproplowerbound} we have
\begin{equation} \label{eqn:constantZerotermLowerbound}
\frac{1}{2\pi} \hat{m}_{\Delta}(0) \log \pi =  \left( \frac 52 - \alpha \right)\log \pi - \frac{\log \pi}{\pi\Delta}\log \left(\frac{1 - e^{-(2\alpha - 1)\pi\Delta}}{1 - e^{-4\pi\Delta}}\right).
\end{equation} 

\subsubsection{Third term}
Proceeding as in \S 2.1.3 we obtain
\begin{align} \label{eqn:integraltermLowerbound}
\begin{split}
& \frac{1}{2\pi}\int_{-\infty}^{\infty} m_{\Delta}(x) \,{\rm Re} \frac{\Gamma '}{\Gamma}\left( \frac{1}{4} + \frac{i(t-x)}{2} \right) \dx \\
& \ \ \ \ \ = \left(\frac 52 - \alpha \right)\log \frac t2 - \frac{1}{\pi\Delta} \log \left( \frac{1 - e^{-(2\alpha - 1)\pi \Delta}}{1 - e^{-4\pi\Delta}}\right) \log \frac t2 \\
& \ \ \ \ \ \ \ \  \ \ \ \ \ \ \  + O\left( \frac{\Delta \log (1 + \sqrt{t}\Delta)}{\sqrt{t}} \right). 
\end{split}
\end{align}

\subsubsection{Fourth term (sum over prime powers)}
By the same arguments that led to (\ref{eqn:mainsumoverprimes}), using inequality (\ref{ineq:twotermdifferent}), the sum over prime is bounded below as follows (recall that $x = e^{2\pi\Delta}$)
\begin{align} \label{eqn:LowerboundMainsumoverprimes}
\begin{split}
\frac{1}{2\pi} \sum_{n = 2}^{\infty} & \frac{\Lambda(n)}{\sqrt{n}}\,\hat{m}_{\Delta} \left( \frac{\log n}{2\pi} \right)\left( e^{-i t\log n} + e^{it \log n }\right) \\
&\geq - 2 \sum_{n \leq x} \frac{\Lambda(n)}{\sqrt{n}} \sum_{k = 0}^{\infty} \left(\frac{k + 1}{k \log x + \log n }  \frac{1}{(nx^k)^{\alpha - 1/2}}  \right. \\
& \ \ \ \ \ \ \ \ \ \ \ \ \ \ \ \ \ \ \ \ \ \ \ \ \ \ \left. - \frac{k + 1}{((k + 2)\log x - \log n)}  \frac{n^{\alpha - 1/2}}{(x^{k + 2})^{\alpha - 1/2}} \right) \\
&  \ \ \ \ \ \ \ \ \ \ \ \ \ \ \ \ \ \ \ \ \ \ \ \ \ \ \ \ \ \ \ \ \ \ \ \ \ \ \ - 2 \log \left|\zeta \left(\frac 52 + it\right)\right| + O\bigl(e^{- 3\pi \Delta}\bigr). 
\end{split}
\end{align} 
The sum over $k$ in (\ref{eqn:LowerboundMainsumoverprimes}) is bounded above by 
\begin{align} \label{ineq:BoundOverkLowerbound}
\begin{split}
\sum_{k = 0}^{\infty} & \left(\frac{k + 1}{k \log x + \log n }  \frac{1}{(nx^k)^{\alpha - 1/2}}  - \frac{k + 1}{((k + 2)\log x - \log n)}  \frac{n^{\alpha - 1/2}}{(x^{k + 2})^{\alpha - 1/2}} \right) \\
&\leq \ \ \sum_{k = 0}^{\infty} \frac{1}{(x^k)^{\alpha - 1/2}} \left(\frac{1}{n^{\alpha - 1/2}\log n}    - \frac{n^{\alpha - 1/2}}{(2\log x - \log n)(x^{2\alpha - 1})} \right) \\
&=  \ \ \frac{x^{\alpha - 1/2}}{x^{\alpha - 1/2} - 1} \left(\frac{1}{n^{\alpha - 1/2}\log n}    - \frac{n^{\alpha - 1/2}}{(2\log x - \log n)(x^{2\alpha - 1})} \right). 
\end{split}
\end{align}
Using partial summation, the prime number theorem (\ref{eqn:PNT}), equation (\ref{ineq:BoundOverkLowerbound}), and the integrals on the Appendix, we obtain that
\begin{align}\label{Fai_2}
\begin{split}
\sum_{n \leq x} & \frac{\Lambda(n)}{\sqrt{n}} \sum_{k = 0}^{\infty}  \left(\frac{k + 1}{k \log x + \log n }  \frac{1}{(nx^k)^{\alpha - 1/2}}  - \frac{k + 1}{((k + 2)\log x - \log n)}  \frac{n^{\alpha - 1/2}}{(x^{k + 2})^{\alpha - 1/2}} \right) \\
&\leq  \frac{x^{\alpha - 1/2}}{x^{\alpha - 1/2} - 1}  \left\{ \int_{2}^{x} \left( \frac{1}{t^{\alpha} \log t} - \frac{1}{x^{2\alpha - 1}}  \frac{1}{t^{1- \alpha}(2\log x - \log t)} \right)  \dt \right.\\
&\left. \ \ \ \ \ \ \ \ \ \ \ \ \ \ \ \ \ \ \ \ \ \ \ \ \ \ \ \ \ \ \ \ \ \ \ \ \ \ \ \ \ \ \ \ \ \ \ \ \ \ \ \ \ \ \ \ \ \ \ \ \ + O\left( \min \left\{ \log x, \frac{1}{\alpha - 1/2} \right\} \right)\right\} \\
&= \left\{ \begin{array}{l} \log \log x  + O(1), \ \ \ \  {\rm if }\,\,\, (1 - \alpha)\log x = O(1); \\
\\
\frac{x^{\alpha - 1/2}}{x^{\alpha - 1/2} - 1}\left\{\frac{2\alpha - 1}{\alpha (1 - \alpha)}\frac{x^{1-\alpha}}{\log x}   + \log \log x + O\left(\frac{x^{1 - \alpha}}{(1- \alpha)^2\log^2 x}\right)\right\}, \ {\rm otherwise.} \end{array}\right. 
\end{split}
\end{align}

\subsubsection{Final Analysis}
Combining the bounds (\ref{Final12})-(\ref{Fai_2}) above, and using the fact that $x = e^{2\pi \Delta}$, we derive
\begin{align} \label{Fai_3}
\begin{split}
&\log |\zeta(\alpha + it)| \geq  \frac{1}{2\pi\Delta} \log \left( \frac{1 - e^{-(2\alpha - 1)\pi \Delta}}{1 - e^{-4\pi\Delta}}\right) \log \frac t2 \\
\\
& \ \ \ \ \ \ \ \ \ \ \ \ \ \ \ \ \ \ \ \ \ \  + O\left( \Delta^2 \frac{e^{\pi \Delta}}{1 + \Delta t} \right) + O\left(\frac{\Delta \log(1 + \sqrt{t}\Delta)}{\sqrt{t}}\right)\\
&\\
& \ \ \ \ - \left\{ \begin{array}{l} \log 2\pi \Delta  + O\left(1\right), \ \  {\rm if }\,\,\, (1 - \alpha)\Delta = O(1);  \\
\\
\left(\frac{e^{(2\alpha - 1)\pi\Delta}}{e^{(2\alpha - 1)\pi\Delta} - 1}\right) \left\{ \frac{2\alpha - 1}{\alpha (1 - \alpha)}\frac{e^{(2-2\alpha)\pi\Delta}}{2\pi\Delta} + \log 2\pi\Delta + O\left(\frac{e^{(2-2\alpha)\pi\Delta}}{(1- \alpha)^2\pi^2\Delta^2}\right)\right\},\\
\ \ \ \ \ \ \ \ \ \ \ \ \ \ \ \ \ \ \ \ \ \ \ \ \ \ \ \ \ \ \  \ \ \ \ \ \ \ \ \ \ \ \ \ \ \ \ \ \ \ \ \ \ \ \ \ \ \ \ \ \ \ \ \ \ \ \ \ \ \   {\rm otherwise.} 
\end{array}\right.
\end{split}
\end{align}
An optimal bound in (\ref{Fai_3}) occurs when $\pi \Delta = \log \log t.$ Similar to the upper bound, the lower bound depends on the location of $\alpha$, and again we examine three cases:
\begin{enumerate}
\item[{\it Case 1}.] $\alpha - 1/2 = O\left(\frac{1}{\log \log t}\right)$.\\
For this case, $\frac{2\alpha - 1}{\alpha (1-\alpha)} = O\left(\frac{1}{\log \log t}\right)$ and the lower bound (\ref{Fai_3}) becomes
\begin{align*}
 \log |\zeta(\alpha + it)|  & \geq  \log \left( 1 -  (\log t)^{1 - 2\alpha }\right) \frac{\log t}{2\log \log t} \\
 & \ \ \ \ \ \ \ \ \ \ \ \ \ \ \ \ \ \ \ \ \ - O\left( \frac{(\log t)^{2-2\alpha}}{(\log \log t)^2(1 - (\log t)^{1 - 2\alpha})} \right).
\end{align*}
Observe that when $\alpha \rightarrow 1/2,$ the bound goes to $-\infty,$ which corresponds to the case when $\zeta(1/2  + it) = 0.$\\

\item[{\it Case 2}.] $1 - \alpha = O\left(\frac{1}{\log \log t}\right).$ \\
For this case, the lower bound (\ref{Fai_3}) is
$$ \log |\zeta(\alpha + it)|  \geq - \log (2 \log \log t) - O(1).$$
In \S 3.2, we will bound explicitly what the constant term is for $|\zeta(1 + it)|.$\\

\item[{\it Case 3}.] Otherwise, $\log \left( 1 -  e^{-(2\alpha - 1)\pi \Delta}\right) \asymp -\frac{1}{(\log t)^{2\alpha - 1} }$, and the lower bound (\ref{Fai_3}) becomes 
\begin{align*}
\log |\zeta(\alpha + it)|  &\geq - \left( \frac{1}{2} + \frac{2\alpha - 1}{\alpha (1 - \alpha)} \right)\frac{(\log t)^{2 - 2\alpha}}{\log \log t} - \log (2\log \log t) \\ 
&  \ \ \ \ \ \ \ \ \ \ \ \ \ \ \ \ \ \ \ \ \ \ \ \ \ \ \ \ \ \ \ \ -O\left(\frac{(\log t)^{2 - 2\alpha}}{(1-\alpha)^2(\log \log t)^2}\right).
\end{align*}
\end{enumerate}
This completes the proof of Theorem \ref{thm:lowerbound}.

\subsection{A lower bound for $|\zeta (1 + it)|$}  In this section we will bound $1/\zeta(1 + it)$ and rederive Littlewood's result \cite{L3},
$$ \frac{1}{|\zeta(1 + it)|} \leq \left(\frac{12e^{\gamma}}{\pi^2} + o(1)\right)\log \log t ,$$
where $\gamma$ is the Euler constant.

To obtain this bound, we will use  (\ref{eqn:upplogeqn}). The method exploited to derive the bound for $1/|\zeta(1 + it)|$ is the same as the one in \S 3.1 except that we will bound ${\rm Re} \sum_{n \leq x} \frac{\Lambda(n)}{n^{1 + it}\log n}$ with the error term $o(1).$ 
The following identity from \cite{L3} is useful in bounding the sum over prime powers:
\begin{eqnarray*}
{\rm Re} \sum_{n \leq x} \frac{\Lambda(n)}{n^{1 + it} \log n} &=& -{\rm Re} \ \log \prod_{p \leq x} \left( 1 - \frac{1}{p^{1 + it}}\right) + O(\frac{1}{\sqrt{x}}) \\
&\geq& -{\rm Re} \ \log \prod_{p \leq x} \left( 1 + \frac{1}{p}\right) + O(\frac{1}{\sqrt{x}}). \\
&=& - \log \left(\frac{6e^{\gamma}}{\pi^2} \log x\right) + o(1). 
\end{eqnarray*}
Moreover by the prime number theorem, we have
$$ \sum_{n \leq x} \frac{\Lambda(n)}{x(2\log x - \log n)} \leq \frac{1}{x \log x} \sum_{n \leq x} \Lambda(n) = O\left(\frac{1}{\log x}\right).$$
By (\ref{eqn:upplogeqn}) and the same arguments contained in \S 3.1 we obtain
\begin{eqnarray*}
\log |\zeta(1 + it)| &\geq& - \log \left(\frac{12e^{\gamma}}{\pi^2} \pi \Delta \right) + \frac{1}{2\pi\Delta}\log \left( \frac{1 - e^{-\pi \Delta}}{1 - e^{-4\pi\Delta}}\right) \log \frac t2 \\
&& \,\,\,\,\,\,\,\,\,\,\,\,\,\,\, + O\left(\frac{\Delta \log (1 + \sqrt{t}\Delta)}{\sqrt{t}} + \frac{\Delta e^{\pi\Delta}}{t} + \frac{1}{\pi \Delta} \right).
\end{eqnarray*}
If we pick $\pi\Delta = \log \log t$, 
we obtain the bound in Corollary \ref{cor:boundAt1}.

\section{Extremal functions}
In this section we will discuss the extremal functions used in this paper, proving Lemmas \ref{prop:mainprop} and \ref{prop:mainproplowerbound}. This study relies substantially on the recent work of Carneiro, Littmann and Vaaler \cite{CLV} that contains the solution of the Beurling-Selberg extremal problem for the Gaussian and a general integration technique on the free parameter, producing a variety of new examples. In particular, the logarithmic family $f_{\alpha}(x)$ considered in this paper falls in the range of the ideas in \cite{CLV}.

Throughout this section we let $a = (\alpha - 1/2)\Delta,$ and $b = 2\Delta$.  We have the following identity
\begin{equation} \label{eqn:measureforfunction}
\log \left( \frac{x^2 + b^2}{x^2 + a^2} \right) = \int_0^{\infty} e^{-\pi \lambda x^2}\left(\frac{e^{-\pi \lambda a^2} - e^{-\pi \lambda b^2}}{\lambda}\right)\, \dl
\end{equation}
Define $F_{\Delta}(x)$ to be the expression on the left hand side of (\ref{eqn:measureforfunction}). It is clear that 
$$ f_\alpha(x) =  F_{\Delta}(\Delta x).$$
By Corollary 17 in \cite{CLV}, there is a unique extremal minorant $G_{\Delta}(x)$ and a unique extremal majorant $M_{\Delta}(x)$ of exponential type $2\pi$ for $F_{\Delta}(x).$ We will let
\begin{equation}\label{defofg}
g_{\Delta}(x) = G_{\Delta}(\Delta x) \,\,\,\,\,\, {\rm and} \,\,\,\,\,\,m_{\Delta}(x) = M_{\Delta}(\Delta x) 
\end{equation}
From \cite{CLV}, we also have
\begin{equation} \label{eqn:sumofGDelta}
G_{\Delta}(z) = \Big( \frac{\cos \pi z}{\pi }\Big)^2 \sum_{n=-\infty}^{\infty} 
 \left\{\frac{F_{\Delta}\bigl(n-\frac 12\bigr)}{\bigl(z-n+\frac 12\bigr)^2} + \frac{F_{\Delta}^{\prime}\bigl(n-\frac 12\bigr)}{\bigl(z-n+\frac 12\bigr)}\right\}, 
\end{equation}
and 
\begin{equation} \label{eqn:sumofMDelta}
M_{\Delta}(z) = \Big( \frac{\sin \pi z}{\pi }\Big)^2 \sum_{n=-\infty}^{\infty} 
 \left\{ \frac{F_{\Delta}(n)}{(z-n)^2} + \frac{F_{\Delta}^{\prime}(n)}{(z-n)}\right\}. 
\end{equation}

\subsection{Proof of Lemma \ref{prop:mainprop}} Part (iii) of Lemma \ref{prop:mainprop} is contained in \cite[Corollary 17, Example 3]{CLV}, and thus we will focus here in proving parts (i) and (ii).
\subsubsection{Part (i)} Observe first that 
\begin{equation} \label{eqn:sumofGDelta1}
G_{\Delta}(z) = \sum_{n=-\infty}^{\infty} \left(\frac{\sin \pi \bigl(z-n + \frac 12\bigr)}{\pi\bigl(z-n + \frac 12\bigr)}\right)^2 
 \left\{f_{\alpha}\left(\frac{n-\tfrac 12}{\Delta}\right) + \frac{\bigl(z-n + \frac 12\bigr)}{\Delta} f_{\alpha}^{\prime}\left(\frac{n-\tfrac 12}{\Delta}\right)\right\}. 
\end{equation}
For any complex number $\xi$ we have $(\sin (\pi \xi)/ (\pi \xi))^2 \ll e^{2 \pi |{\rm Im} \xi|}/ (1 + |\xi|^2)$. Using the fact that 
\begin{equation}\label{relationforf}
 f_{\alpha}(x) \leq \frac{4}{x^2 + \bigl(\alpha - \tfrac 12\bigr)^2} \ \ \ \  \textrm{and} \ \ \ \ f_{\alpha}^{\prime}(x) \leq \frac{8|x|}{(x^2 + 4)\bigl(x^2 + \bigl(\alpha - \tfrac 12\bigr)^2\bigr)}\,,
\end{equation}
we can split the sum (\ref{eqn:sumofGDelta1}) in two parts, where $n \leq |z|/2$ and $n \geq |z|/2$, to conclude that 
\begin{equation*}
 |G_{\Delta}(x+iy)| \ll  \frac{\Delta^2  }{1+ |x+iy|}e^{2\pi |y|}.
\end{equation*}
and from (\ref{defofg}) we arrive at (\ref{minlemma2}). 

For $x$ real, we have $f_{\alpha}(x)\geq 0$ and $f_{\alpha}^{\prime}(-x) = -f_{\alpha}^{\prime}(x)$, and we can pair the terms $n\geq 1$ and $1-n\leq0$ in the sum (\ref{eqn:sumofGDelta1}) to obtain
\begin{align}\label{2boundforg}
\begin{split}
G_{\Delta}(x) &\geq \Big( \frac{\cos \pi x}{\pi }\Big)^2 \sum_{n=1}^{\infty} \frac{1}{\Delta} f_{\alpha}^{\prime}\left(\frac{n-\tfrac 12}{\Delta}\right)\left\{ \frac{1}{\bigl(x - n + \tfrac 12\bigr)} - \frac{1}{\bigl(x + n - \tfrac12\bigr)}\right\}\\
& = \sum_{n=1}^{\infty} \frac{\sin^2 \pi\bigl(x -n +\tfrac12\bigr)}{\pi^2 \bigl(x^2 - \bigl(n-\tfrac 12\bigr)^2\bigr)}\, \frac{2 \bigl( n - \tfrac12 \bigr)}{\Delta}\, f_{\alpha}^{\prime}\left(\frac{n-\tfrac 12}{\Delta}\right).
\end{split}
\end{align}
Using (\ref{relationforf}) and (\ref{2boundforg}) we can show that there is a constant $C$ such that 
\begin{equation*}
-C \frac{\Delta^2}{\Delta^2 + x^2} \leq G_{\Delta}(x)\,,
\end{equation*}
and thus from (\ref{defofg}) we arrive at (\ref{minlemma21}), completing the proof of part (i).

\subsubsection{Part (ii)}  It is sufficient to consider the Fourier transform of $G_{\Delta}(x)$ since $\hat{g}_{\Delta}(y) = \frac{1}{\Delta}\hat{G}_{\Delta}\big(\frac{y}{\Delta}\big).$  

For $|y| \geq 1,$  $\hat{G}_{\Delta}(y) = 0.$ Therefore, in what follows we will consider $\hat{G}_{\Delta}(y)$ when $|y| < 1.$ From (\ref{eqn:measureforfunction}) and \cite[Theorem 4]{CLV}, we know that
\begin{align}\label{bigintegral}
\begin{split}
\hat{G}_{\Delta}(y) & = \int_0^{\infty} \left\{(1 - |y|) \sum_{n = -\infty}^{\infty} e^{-\pi \lambda (n + 1/2)^2} e^{2\pi i y(n + 1/2)} \right.\\ 
& \ \ \ \ \ \ \ \ \ \ \ \ \ \ \ \ \left. - \frac{\lambda}{2\pi}{\rm sgn}(y) \sum_{n = -\infty}^{\infty} 2\pi i \bigl(n + \tfrac 12\bigr) e^{-\pi \lambda (n + 1/2)^2} e^{2\pi i y(n + 1/2)} \right\} \\
&\ \ \ \ \ \ \ \ \ \ \ \ \ \ \ \ \ \ \ \ \ \ \cdot \left( \frac{e^{-\pi \lambda a^2} - e^{-\pi \lambda b^2}}{\lambda}\right) \dl.
\end{split}
\end{align}
It is easy to see that $\hat{G}_{\Delta}(y)$ is an even function. Therefore it is sufficient to consider the case $0 \leq y < 1$. We will evaluate the integrals of the first and second sums separately.
\\
{\it Integration of the first sum.} By calculus, we can show that 
\begin{align} \label{firstIntMin}
\begin{split}
 \int_0^{\infty}& (1 - |y|) \left\{\sum_{n = -\infty}^{\infty} e^{-\pi \lambda (n + 1/2)^2} e^{2\pi i y(n + 1/2)} \right\} \cdot \left( \frac{e^{-\pi \lambda a^2} - e^{-\pi \lambda b^2}}{\lambda}\right) \dl  \\
&= (1 - |y|) \, e^{\pi iy}\sum_{n = -\infty}^{\infty} \log\left( \frac{(n + 1/2)^2 + b^2}{(n + 1/2)^2 + a^2}\right)e^{2\pi iyn}. 
\end{split}
\end{align}
Let 
\begin{equation*}
\hat{k}(x) = \log\left( \frac{(x + 1/2)^2 + b^2}{(x + 1/2)^2 + a^2}\right).
\end{equation*}
To evaluate the sum over $n$, we will use Poisson summation formula, 
$$ \sum_{n \in \mathbb{Z}} \hat{k}(n) e^{2\pi i yn} = \sum_{n \in \mathbb{Z}} k(y + n).$$
Therefore we need to compute $k(w)$. For $w\neq 0$ we use integration by parts to get

\begin{align}\label{Sec3.2}
\begin{split}
k(w) &= \int_{-\infty}^{\infty} \log\left( \frac{(x + 1/2)^2 + b^2}{(x + 1/2)^2 + a^2}\right) e^{2\pi iwx} \, \dx \\
& \ \ \ \ \ \ \ \ \ \ \ \ \ \ \ \ \ \ = \frac{e^{-\pi i w}}{2\pi i w}  \int_{-\infty}^{\infty}  \frac{2x(b^2 - a^2)}{(x^2 + b^2)(x^2 + a^2)}\,e^{2\pi i wx}\, \dx. 
\end{split}
\end{align}
For $w = 0$ we will have
\begin{eqnarray*}
k(0) = \int_{-\infty}^{\infty} \log\left( \frac{(x + 1/2)^2 + b^2}{(x + 1/2)^2 + a^2}\right) \, \dx 
= 2\pi (b-a). 
\end{eqnarray*}
The integrals above can be computed via contour integration.
\begin{enumerate}
\item[{\it Case 1:}] $w > 0.$ The chosen contour is a rectangle with vertices $-X, X, X + iY, -X + iY,$ where $X, Y > 0,$ and $X, Y \rightarrow \infty.$
Therefore
\begin{equation}\label{Sec3.4}
\int_{-\infty}^{\infty}  \frac{2x(b^2 - a^2)}{(x^2 + b^2)(x^2 + a^2)}\,e^{2\pi i wx}\, \dx = 2\pi i \left( e^{-2\pi wa} - e^{-2\pi wb}\right).
\end{equation}
\item[{\it Case 2:}] $w < 0.$ The contour is a rectangle with vertices $X, -X, -X - iY, X - iY,$  where $X, Y > 0,$ and $X, Y \rightarrow \infty.$
Therefore
\begin{equation}\label{Sec3.5}
 \int_{-\infty}^{\infty}  \frac{2x(b^2 - a^2)}{(x^2 + b^2)(x^2 + a^2)}\,e^{2\pi i wx}\, \dx = -2\pi i \left( e^{2\pi wa} - e^{2\pi wb}\right).
\end{equation}
\end{enumerate}
Combining equations (\ref{firstIntMin}) - (\ref{Sec3.5}) above, for $y \neq 0,$ we obtain 
\begin{align*} 
& \int_0^{\infty} (1 - |y|) \left\{\sum_{n = -\infty}^{\infty} e^{-\pi \lambda (n + 1/2)^2} e^{2\pi i y(n + 1/2)} \right\} \cdot \left( \frac{e^{-\pi \lambda a^2} - e^{-\pi \lambda b^2}}{\lambda}\right)\,\dl \\
&= (1 - |y|)  \left\{ \sum_{n = 0}^{\infty} (-1)^n \frac{e^{-2\pi (y + n)a} - e^{-2\pi (y + n) b}}{y + n} \right.\\
& \ \ \ \ \ \ \ \ \ \ \ \ \ \ \ \ \ \ \ \ \ \ \ \ \ \ \ \ \ \ \ \ \ \ \ \ \ \ \ \ \ \ \ \ \ \ \ \left.- \sum_{n = 1}^{\infty} (-1)^n \frac{e^{2\pi (y - n)a} - e^{2\pi (y - n) b}}{y - n}\right\}.
\end{align*}
For $y = 0,$ the integral is
$$ 2\pi (b -a) - 2\log \left(\frac{1 + e^{-2\pi a}}{1 + e^{-2\pi b}}\right).$$
\\
{\it Integration of the second sum}. For $y = 0$, the second integral is 0. So we will compute its value for $0<y<1$. By calculus, we have
\begin{align}\label{Sec3.6}
\begin{split}
& -\int_{0}^{\infty} \frac{\lambda}{2\pi}\left\{\sum_{n = -\infty}^{\infty} 2\pi i \bigl(n + \hh\bigr) e^{-\pi \lambda (n + 1/2)^2} e^{2\pi i y(n + 1/2)} \right\} \left( \frac{e^{-\pi \lambda a^2} - e^{-\pi \lambda b^2}}{\lambda}\right)\dl \\
&= -i \frac{e^{\pi iy}}{\pi} \sum_{n = -\infty}^{\infty} \left( \frac{(n + 1/2)}{(n + 1/2)^2 + a^2} - \frac{(n + 1/2)}{(n + 1/2)^2 + b^2} \right)e^{2\pi i yn}.
\end{split}
\end{align}
Let 
\begin{equation*}
\hat{h}(x) = \frac{(x + 1/2)}{(x + 1/2)^2 + a^2} - \frac{(x + 1/2)}{(x + 1/2)^2 + b^2}.
\end{equation*} 
Again we will compute the sum above by Poisson summation formula,
\begin{equation}\label{Sec3.7}
 \sum_{n \in \mathbb{Z}} \hat{h}(n) e^{2\pi i yn} = \sum_{n \in \mathbb{Z}} h(y + n).
\end{equation}
Since $y + n \neq 0,$ we will compute $h(w),$ where $w \neq 0$,
\begin{align}\label{Sec3.8}
\begin{split}
h(w) &= \int_{-\infty}^{\infty} \left( \frac{(x + 1/2)}{(x + 1/2)^2 + a^2} - \frac{(x + 1/2)}{(x + 1/2)^2 + b^2} \right) e^{2\pi iwx}\, \dx \\
&= e^{-\pi i w} \int_{-\infty}^{\infty} \left( \frac{x}{x^2 + a^2} - \frac{x}{x^2 + b^2} \right) e^{2\pi iwx}\, \dx
\end{split}
\end{align}
We now use contour integration again.
\begin{enumerate}

\item[{\it Case 1:}] $w > 0$. The contour is a rectangle with vertices $-X, X, X + iY, -X + iY,$ where $X, Y > 0,$ and $X, Y \rightarrow \infty.$
Therefore 
\begin{align}\label{Sec3.9}
\begin{split}
h(w) &= 2\pi i  e^{-\pi i w} \left( {\rm res}_{x = ia} \frac{x}{x^2 + a^2} \cdot e^{2\pi iwx} - {\rm res}_{x = ib} \frac{x}{x^2 + b^2} \cdot e^{2\pi iwx} \right) \\
&= \pi i  e^{-\pi i w} \left( e^{-2\pi wa} - e^{-2\pi wb}\right).
\end{split}
\end{align}

\item[{\it Case 2:}] $w < 0$. The contour is a rectangle with vertices $X, -X, -X - iY, X - iY,$ where $X, Y > 0,$ and $X, Y \rightarrow \infty.$
Therefore 
\begin{align}\label{Sec3.10}
\begin{split}
h(w) &= - 2\pi i  e^{-\pi i w} \left( {\rm res}_{x = -ia} \frac{x}{x^2 + a^2} \cdot e^{2\pi iwx} - {\rm res}_{x = -ib} \frac{x}{x^2 + b^2} \cdot e^{2\pi iwx} \right) \\
&= - \pi i  e^{-\pi i w} \left( e^{2\pi wa} - e^{2\pi wb}\right).
\end{split}
\end{align}
\end{enumerate}
Finally, combining (\ref{Sec3.6})- (\ref{Sec3.10}) we obtain
\begin{align*}
& -\int_{0}^{\infty} \frac{\lambda}{2\pi}\left\{\sum_{n = -\infty}^{\infty} 2\pi i (n + \h) e^{-\pi \lambda (n + 1/2)^2} e^{2\pi i y(n + 1/2)} \right\} \left( \frac{e^{-\pi \lambda a^2} - e^{-\pi \lambda b^2}}{\lambda}\right) \dl  \\
&= \sum_{n = 0}^{\infty} (-1)^{n}  \left( e^{-2\pi (y + n)a} - e^{-2\pi (y + n) b}\right) -  \sum_{n = 1}^{\infty} (-1)^{n} \left( e^{2\pi (y - n) a} - e^{2\pi (y - n)b}\right),
\end{align*}
and this ultimately leads to part (ii) of Lemma \ref{prop:mainprop}.

\subsection{Proof of Lemma \ref{prop:mainproplowerbound}} The proof of part (i) of Lemma \ref{prop:mainproplowerbound} is very similar to the analogous part (i) of Lemma \ref{prop:mainprop}, proved in \S4.1. Part (iii) of Lemma \ref{prop:mainproplowerbound} is contained in \cite[Corollary 17, Example 3]{CLV}, and thus we will only focus here on part (ii).

\subsubsection{Part (ii)} Since $M_{\Delta}(y)$ is an even function, it suffices to consider $\hat{M}_{\Delta}(y)$ for $0 \leq y < 1$. We know from (\ref{eqn:measureforfunction}) and \cite[Theorem 4]{CLV} that
\begin{align}
\begin{split}\label{Sec3.15}
\hat{M}_{\Delta}(y) &= \int_0^{\infty} \left\{(1 - |y|) \sum_{n = -\infty}^{\infty} e^{-\pi \lambda n^2} e^{2\pi i yn} - \frac{\lambda}{2\pi}{\rm sgn}(y) \sum_{n = -\infty}^{\infty} 2\pi in \, e^{-\pi \lambda n^2} e^{2\pi i yn} \right\} \\
& \ \ \ \ \ \ \ \ \ \ \ \ \ \ \ \cdot \left( \frac{e^{-\pi \lambda a^2} - e^{-\pi \lambda b^2}}{\lambda}\right) \dl.
\end{split}
\end{align}
{\it Integration of the first sum.} By the same arguments used for $G_{\Delta}(y),$ we have
\begin{align}\label{Sec3.16}
\begin{split}
 &\int_0^{\infty} (1 - |y|) \left\{\sum_{n = -\infty}^{\infty} e^{-\pi \lambda n^2} e^{2\pi i yn} \right\} \left( \frac{e^{-\pi \lambda a^2} - e^{-\pi \lambda b^2}}{\lambda}\right) \dl \\
&  =(1 - |y|) \sum_{n = -\infty}^{\infty} \log\left( \frac{n^2 + b^2}{n^2 + a^2}\right)e^{2\pi iyn}\\
& = (1 - |y|)  \left\{ \sum_{n = 0}^{\infty}  \frac{e^{-2\pi (y + n)a} - e^{-2\pi (y + n) b}}{y + n} - \sum_{n = 1}^{\infty}  \frac{e^{2\pi (y - n)a} - e^{2\pi (y - n) b}}{y - n}\right\}\,,
\end{split}
\end{align}
for $y \neq 0$. For $y = 0$ we have the value
$$ 2\pi (b -a) - 2\log \left(\frac{1 - e^{-2\pi a}}{1 - e^{-2\pi b}}\right).$$
{\it Integration of the second sum.} By the same arguments used for $G_{\Delta}(y)$, the second term is equal to 
\begin{align}
\begin{split}\label{Sec3.17}
& -\int_{0}^{\infty} \frac{\lambda}{2\pi}\left\{\sum_{n = -\infty}^{\infty} 2\pi i n \,e^{-\pi \lambda n^2} e^{2\pi i yn} \right\}\left( \frac{e^{-\pi \lambda a^2} - e^{-\pi \lambda b^2}}{\lambda}\right) \dl \\
&= \sum_{n = 0}^{\infty}  \left( e^{-2\pi (y + n)a} - e^{-2\pi (y + n) b}\right) -  \sum_{n = 1}^{\infty}  \left( e^{2\pi (y - n) a} - e^{2\pi (y - n)b}\right).
\end{split}
\end{align}
Combining (\ref{Sec3.15}), (\ref{Sec3.16}) and (\ref{Sec3.17}) we complete the proof of part (ii).

\section{Appendix}
Here we show the following asymptotics:
\begin{enumerate}
\item[{\bf A1}.]
\begin{equation} \label{eqn:asymforintlambdanNlogn1}
 \int_{2}^{x} \frac{1}{t^{\alpha}\log t} \ \dt = \left\{ \begin{array}{l} \log \log x  + O(1),\ \ \  {\rm if} \,\,\,\, (1 - \alpha)\log x = O(1); \\ 
\frac{x^{1 - \alpha}}{(1 - \alpha)\log x} + \log \log x  + O\left(\frac{x^{1 - \alpha}}{(1 - \alpha)^2\log^2 x} \right),\ {\rm otherwise}. \end{array} \right.
\end{equation}
\item [{\bf A2}.]
\begin{equation} \label{eqn:asymforthirdterm} 
\int_2^x \frac{1}{t^{1 - \alpha}(2\log x - \log t)} \ \dt = \frac{1}{\alpha} \frac{x^{\alpha}}{\log x} + O\left( \frac{x^{\alpha}}{\log^2 x}\right).
\end{equation}
\end{enumerate}

\begin{proof}[Proof of A1.] The left hand side of (\ref{eqn:asymforintlambdanNlogn1}) can be written as
\begin{eqnarray} \label{Fai_night}
\int_{2}^{x} \frac{1}{t^{\alpha}\log t} \ \dt &=& \log \log x - \log \log 2 + \int_{(1 - \alpha)\log 2}^{(1 - \alpha)\log x} \frac{e^y - 1}{y} \ \dy.
\end{eqnarray}
If $(1 - \alpha)\log x = O(1),$ then $\frac{e^y - 1}{y} \leq e^{y} = O(1)$ for $(1 - \alpha)\log 2 \leq y \leq (1 - \alpha) \log x.$ Therefore the integral on the right hand side of (\ref{Fai_night}) is $O(1).$

Otherwise, the integral on the right hand side of (\ref{Fai_night}) is $O(1)$ plus
\begin{align*}
 \frac{e^y - y - 1}{y} & \big|^{(1- \alpha) \log x}_0 + \int_0^{(1- \alpha) \log x} \frac{e^y - y - 1}{y^2} \ \dy\\
&  = \frac{x^{1 - \alpha}}{(1- \alpha) \log x} + \int_0^{(1- \alpha) \log x} \frac{e^y - y - 1}{y^2} \ \dy + O(1).
\end{align*}
Since $\lim_{y \rightarrow 0} \frac{e^y - y - 1}{y^2} = \frac{1}{2},$ 
\begin{align*}
& \int_0^{(1- \alpha) \log x} \frac{e^y - y - 1}{y^2} \ \dy \\
&= \int_2^{\tfrac{(1 - \alpha)\log x}{2}} \frac{e^y - y - 1}{y^2} \ \dy + \int_{\tfrac{(1 - \alpha)\log x}{2}}^{(1 - \alpha)\log x} \frac{e^y - y - 1}{y^2} \ \dy + O(1) \\
&\leq \,\, \frac{1}{4}  \int_2^{\tfrac{(1 - \alpha)\log x}{2}} \bigl(e^y - y - 1\bigr) \dy + \frac{4}{(1 - \alpha)^2\log^2 x} \int_{\tfrac{(1 - \alpha)\log x}{2}}^{(1 - \alpha)\log x} \bigl(e^y - y - 1\bigr)  \dy + O(1) \\
&= O\left(\frac{x^{1 - \alpha}}{(1- \alpha)^2 \log^2 x} \right).
\end{align*}
\end{proof}

\begin{proof} [Proof of A2.] 

Let $y = x^{2}/t.$ The integral (\ref{eqn:asymforthirdterm}) becomes 
\begin{equation*}
x^{2\alpha} \int_{x}^{x^{2}/2} \frac{1}{y^{1 + \alpha} \log y} \, \dy = \frac{1}{\alpha} \frac{x^{\alpha}}{\log x} -  \frac{x^{2\alpha}}{\alpha} \int_{x}^{x^{2}/2} \frac{1}{y^{1 + \alpha} \log^2 y} \, \dy  + O\left(\frac{1}{\log x}\right).
\end{equation*} 
Expression (\ref{eqn:asymforthirdterm}) follows from the fact that
\begin{align*}
 \int_{x}^{x^{2}/2} \frac{1}{y^{1 + \alpha} \log^2 y} \ \dy  \ll \frac{1}{\log^2 x}  \int_{x}^{x^{2}/2} \frac{1}{y^{1 + \alpha}} \ \dy \ll \frac{1}{x^{\alpha} \log^2 x}.
\end{align*}

\end{proof}

\section*{Acknowledgments}
The authors would like to thank K. Soundararajan for the very valuable discussions during the preparation of this manuscript. This material is based upon work supported by the Institute for Advanced Study and the National Science Foundation under agreement No. DMS-0635607. E. Carneiro would also like to acknowledge support from the Capes/Fulbright grant BEX 1710-04-4.

\end{document}